\documentclass[11pt]{article}
\usepackage{amssymb,amsfonts,amsmath, psfrag}
\usepackage{graphicx, colordvi, cite}
\usepackage{mathrsfs}
\usepackage{rotating}
\usepackage{hyperref}
\usepackage{lmodern}
\usepackage{bm}
\usepackage{amsthm}
\usepackage{url}
\newtheorem{theo}{Theorem}
\newtheorem{prop}[theo]{Proposition}

\newtheorem{lem}[theo]{Lemma}

\makeatletter \@addtoreset{equation}{section}
\@addtoreset{theo}{section} \makeatother

\usepackage{geometry}

\begin{document}
\begin{center}
{\bf \large On a question about pattern avoidance of cyclic permutations}
\vspace{0.3cm}

\end{center}
\begin{center}

 Zuo-Ru Zhang, \ Hongkuan Zhao\\[8pt]
School of Mathematical Science\\
Hebei Normal University\\
 Shijiazhuang 050024, P. R. China\\[6pt]

{\tt  zrzhang@hebtu.edu.cn}

\end{center}

\vspace{0.3cm} \noindent{\bf Abstract.}
Recently, Archer et al.\ studied cyclic permutations that avoid the decreasing pattern $\delta_k=k(k-1)\cdots21$ in one-line notation and avoid another pattern $\tau$ of length $4$ in all their cycle forms.
There are three cases in total to consider, namely, $\tau=1324, 1342$ and $1432$.
They determined two of them, leaving the case $\tau=1432$ as an open question.
In this paper, we resolve this case by deriving explicit formulas based on an analysis of the structure of cycle forms and an application of Dilworth's theorem.

\vspace{0.3cm} 

\noindent {\bf AMS Classification 2020:}
05A05, 05A15
\section{Introduction}
Let $\mathfrak{S}_n$ denote the symmetric group on $[n]=\{1,2,\ldots,n\}$.
Write $\pi\in\mathfrak{S}_n$ in one-line notation as $\pi=\pi_1\pi_2\cdots\pi_n$, where $\pi_i=\pi(i)$.
Given $\sigma=\sigma_1\sigma_2\cdots\sigma_k\in\mathfrak{S}_k$, we say that $\pi$ \emph{contains} $\sigma$ if there exist indices
$i_1<i_2<\cdots<i_k$ such that the subsequence $\pi_{i_1}\pi_{i_2}\cdots\pi_{i_k}$ has the same relative order as $\sigma$.
Otherwise, $\pi$ \emph{avoids} $\sigma$.
See \cite{Knuth,Simion} for more details.

A permutation $\pi\in\mathfrak{S}_n$ is \emph{cyclic} if it is composed of a single $n$-cycle when written as a product of disjoint cycles.
In this case we may write $\pi=(c_1,c_2,\ldots,c_n)$ in cycle notation, and any cyclic rotation represents the same permutation.
We fix the representative that starts with $1$ and call it the \emph{standard cycle form}, $C(\pi)=(1,c_2,\ldots,c_n)$.
For $1\le i\le n$, we say that the cyclic rotation $(c_i,c_{i+1},\ldots,c_n,c_1,\ldots,c_{i-1})$ avoids a pattern $\sigma$ if $c_ic_{i+1}\cdots c_nc_1\cdots c_{i-1}$ avoids $\sigma$.
See \cite{Callan,Vella} for more details.

In \cite{Archer1}, Archer et al.\ studied cyclic permutations under simultaneous pattern restrictions in one-line notation and in the standard cycle form.
In \cite{Archer2}, they considered the decreasing pattern $\delta_k:=k(k-1)\cdots21$ in one-line notation together with another avoidance condition: all cycle forms (i.e. all cyclic rotations of the standard cycle form) avoid a given pattern $\tau$ of length $4$.
Up to symmetry and cyclic rotation, there are three cases to consider.
They enumerated the cases $\tau=1324$ and $\tau=1342$, leaving the case $\tau=1432$ open.
Related conjectures and special cases were studied by Pan \cite{PanJ} and Pan \cite{PanQ}, leading to Pell, Tetranacci and Padovan enumerations in several cases.
The purpose of this paper is to settle the open case $\tau=1432$.

Let $\mathcal{A}_{n}^\circ(\delta_k;1432)$ denote the set of cyclic permutations $\pi \in \mathfrak{S}_n$ such that the one-line notation avoids $\delta_k$ and all cyclic rotations of $C(\pi)$ avoid $1432$.
Set $a_n^\circ(\delta_k;1432)=|\mathcal{A}_{n}^\circ(\delta_k;1432)|$.
Furthermore, define
\[
\mathcal{A}_{n,j}^\circ(\delta_k;1432)=\{\pi \in \mathcal{A}_n^\circ(\delta_k;1432) \mid C(\pi)=(1,j,c_3,\ldots,c_n)\}.
\]
Denote $a_{n,j}^\circ(\delta_k;1432)= |\mathcal{A}_{n,j}^\circ(\delta_k;1432)|$.
Then
\begin{equation}\label{eq1.1}
a_n^\circ(\delta_k;1432)=\sum_{j=2}^n a_{n,j}^\circ(\delta_k;1432).
\end{equation}
For $k \in \{3,4\}$, we apply a standard consequence of Dilworth's theorem \cite{Dilworth} when needed to exclude occurrences of $\delta_k$ in one-line notation.
For $k\geq5$, it turns out that the one-line restriction becomes redundant.
We state our results in the form of the following three theorems:
\begin{theo}\label{th1.1}
For $n\ge 5$, there holds $a_n^\circ(\delta_3;1432)=\lfloor \frac{n^2+7}{2} \rfloor-2n$ (OEIS A061925).
\end{theo}
\begin{theo}\label{th1.2}
For $n\ge 5$, $a_n^\circ(\delta_4;1432)=2^{n-1}-\lfloor \frac{3n-5}{2} \rfloor$.
\end{theo}
\begin{theo}\label{th1.3}
For $n\ge 5$, $a_n^\circ(\delta_k;1432)=2^{n}+1-2n-\binom{n}{3}$,\quad $k\ge5$ (OEIS A088921).
\end{theo}

\section{Proof of Theorem \ref{th1.1}}
To facilitate the presentation, we first introduce the following proposition.
\begin{prop}
All cycle forms of $\pi$ avoid $1432$ if and only if the standard cycle form $C(\pi)$ avoids $321$ and $2143$.
\end{prop}

\noindent \textbf{\textit{Proof}}
It suffices to show that some cyclic rotation of $C(\pi)$ contains a $1432$ pattern if and only if $C(\pi)$ contains a $321$ or $2143$ pattern.

Suppose that a cyclic rotation of $C(\pi)$ has an occurrence of $1432$.
Write it as
\[
(\ldots,x_1,\ldots,x_4,\ldots,x_3,\ldots,x_2,\ldots),\ \  x_4>x_3>x_2>x_1.
\]
Then $C(\pi)$ is one of the following five forms:
\begin{align*}
&(1,\ldots,x_1,\ldots,x_4,\ldots,x_3,\ldots,x_2,\ldots),\\
&(1,\ldots,x_4,\ldots,x_3,\ldots,x_2,\ldots,x_1,\ldots), \\
&(1,\ldots,x_3,\ldots,x_2,\ldots,x_1,\ldots,x_4,\ldots),\\
&(1,\ldots,x_2,\ldots,x_1,\ldots,x_4,\ldots,x_3,\ldots),\\
&(1=x_1,\ldots,x_4,\ldots, x_3,\ldots,x_2,\ldots).
\end{align*}
In the first, second and fifth cases, $x_4x_3x_2$ forms a $321$ pattern.
In the third case, $x_3x_2x_1$ forms a $321$ pattern.
In the fourth case, $x_2x_1x_4x_3$ forms a $2143$ pattern.
Therefore $C(\pi)$ contains a $321$ pattern or a $2143$ pattern.

Conversely, suppose $C(\pi)$ has an occurrence of $2143$.
Write
\[
C(\pi)=(1,\ldots,y_2,\ldots,y_1,\ldots, y_4,\ldots,y_3,\ldots), \quad  y_4>y_3>y_2>y_1.
\]
Then the cyclic rotation $(y_1,\ldots,y_4,\ldots,y_3,\ldots,1,\ldots \allowbreak,y_2,\ldots)$ contains $1432$ pattern $y_1y_4y_3y_2$.
Now suppose $C(\pi)$ has an occurrence of $321$.
Write $C(\pi)$ as
\[
C(\pi)=(1,\ldots,z_3,\ldots,z_2,\ldots,z_1,\ldots), \quad  z_3>z_2>z_1.
\]
Then $1z_3z_2z_1$ forms a $1432$ pattern.
$\hfill\square$

The following lemmas are essential for the proof of Theorem \ref{th1.1} and Theorem \ref{th1.2}.
\begin{lem}\label{lem2.2}
Let $\pi \in \mathfrak{S}_n$ be cyclic with $n\geq5$ and $j=c_2\geq3$.
Then $C(\pi)$ avoids both $321$ and $2143$ if and only if
\begin{enumerate}
\item[(1)] the letters $2,3,\ldots,j-1$ occur in this order in $C(\pi)$;
\item[(2)] the letters greater than $j$ that occur between $j$ and $j-1$ are in increasing order;
\item[(3)] the letters greater than $j$ that occur after $2$ are in increasing order.
\end{enumerate}
\end{lem}
\noindent \textbf{\textit{Proof}}
For the ``only if'' direction, since $C(\pi)$ avoids $321$ and $j=c_2$, the letters $2,3,\ldots,j-1$ appear in increasing order.
Then condition (2) holds for the same reason.
Now, suppose the letters greater than $j$ after $2$ are not in increasing order.
Then there exist $s,t$ after $2$ with $s>t>j$.
In this case, $C(\pi)$ contains the $2143$ pattern $j2st$, a contradiction.

Conversely, to see that $C(\pi)$ avoids $321$, it suffices to show that for any $m \in [n]$, the letters less than $m$ after $m$ appear in increasing order.
This is clear for $m\leq j$.
If $m \in [j+1,n]$ occurs between $j$ and $2$, then all the letters in $[j+1,m-1]$ after $m$ must occur after $j-1$ in increasing order.
So the letters less than $m$ after $m$ appear in increasing order.
If $m \in [j+1,n]$ lies after $2$, then all the letters less than $m$ after $m$ belong to $[3,j-1]$, which appear in increasing order.
Now suppose that $C(\pi)$ contains a $2143$ pattern.
Let $p>q$ play the roles of ``4'' and ``3'', respectively.
Since the letters $1,2,\ldots,j-1$ appear in increasing order, they can only play the role of ``1''.
Hence $p>q>j$.
Therefore, $p$ occurs between $j$ and $2$ by condition (3).
Then the letters before $p$ are increasing by condition (2), contradicting the existence of a $2143$ pattern.
$\hfill\square$

\begin{lem}\label{lem2.3}
For $n\ge5$ and $k\ge3$, we have
\[
a_{n,2}^\circ(\delta_k;1432)=a_{n-1}^\circ(\delta_k;1432).
\]
\end{lem}

\noindent \textbf{\textit{Proof}}
We establish a bijection between the sets
\[
\mathcal{A}_{n,2}^\circ(\delta_k;1432)
=\{\pi\in\mathcal{A}_n^\circ(\delta_k;1432)\mid C(\pi)=(1,2,c_3,\ldots,c_n)\}
\]
and $\mathcal{A}_{n-1}^\circ(\delta_k;1432)$.

Let $\pi\in\mathcal{A}_{n,2}^\circ(\delta_k;1432)$ with $C(\pi)=(1,2,c_3,\ldots,c_n)$.
Define $\pi'\in\mathfrak{S}_{n-1}$ by
\[
C(\pi')=(1,c_3-1,\ldots,c_n-1).
\]
Since $1$ and $2$ are the first two letters of $C(\pi)$ and are the two smallest values, they cannot appear in an occurrence of $321$ or $2143$ in $C(\pi)$.
Hence $C(\pi)$ avoids $321$ and $2143$ if and only if the sequence $1c_3\cdots c_n$ (after standardization) avoids them. Namely, $C(\pi')$ avoids $321$ and $2143$.

For one-line condition, note that the first entry $\pi_1=2$ cannot appear in a decreasing subsequence of length $k\geq3$.
Thus the sequence obtained from the one-line notation of $\pi$ by deleting $2$ and then decreasing all remaining entries except $1$ by $1$ avoids $\delta_k$.
This sequence is precisely the one-line notation of $\pi'$.
Therefore $\pi'\in\mathcal{A}_{n-1}^\circ(\delta_k;1432)$.

Conversely, let $\pi'\in\mathcal{A}_{n-1}^\circ(\delta_k;1432)$ with
$C(\pi')=(1,c_2,\ldots,c_{n-1})$.
Define $\pi$ by
\[
C(\pi)=(1,2,c_2+1,\ldots,c_{n-1}+1).
\]
By the same reasoning as above, inserting $2$ immediately after $1$ and increasing the other letters by $1$ (while fixing $1$) in the standard cycle form cannot create an occurrence of $321$ or $2143$.
Also, inserting $2$ as the first entry and increasing the others except $1$ by $1$ in one-line notation cannot create an occurrence of $\delta_k$.
Hence $\pi\in\mathcal{A}_{n,2}^\circ(\delta_k;1432)$.

The two constructions are inverse to each other, so they give a bijection and the equality follows.
$\hfill\square$

For any $\pi \in \mathfrak{S}_n$, define a binary relation $\le_{\pi}$ on $[n]$ by letting $i \le_{\pi} j$ if and only if $i \le j$ and $\pi_i \le \pi_j$.
Denote $\mathcal{S}_{\pi}$ by the poset $([n],\le_{\pi})$.
It is not hard to see that $\mathcal{S}_{\pi}:=([n],\le_{\pi})$ is a partially ordered set.
Moreover, an antichain in $\mathcal{S}_{\pi}$ of size $k$ is exactly the set of indices of a decreasing subsequence of length $k$ in the one-line notation of $\pi$.
\begin{lem}\label{lem2.4}
Let $\pi \in \mathfrak{S}_n$ and $k\geq3$. If $\mathcal{S}_{\pi}$ is covered by $k-1$ chains, then $\pi$ avoids $\delta_k$ in one-line notation.
\end{lem}
\noindent \textbf{\textit{Proof}}
This is a special case of Dilworth's theorem \cite{Dilworth}, which says that, in any finite poset, the maximum size of an antichain equals the minimum number of chains needed to cover all elements.
$\hfill\square$

Now we consider cyclic permutations that avoid $\delta_3$ in one-line notation and avoid both $321$ and $2143$ in standard cycle form.
We will enumerate permutations $\pi \in \mathcal{A}_n^\circ(\delta_3;1432)$ with $C(\pi)=(1,j,c_3,\ldots,c_n)$ according to the value of $j=c_2$.
\begin{prop}\label{prop2.5}
For $n\geq5$, $a_{n,2}^\circ(\delta_3;1432)=a_{n-1}^\circ(\delta_3;1432).$
\end{prop}
\noindent \textbf{\textit{Proof}}
This follows immediately from Lemma \ref{lem2.3}.
$\hfill\square$

\begin{lem}\label{lem2.6}
For $n\geq5$, $a_{n,3}^\circ(\delta_3;1432)=\lfloor \frac{n-1}{2} \rfloor$.
\end{lem}
\noindent \textbf{\textit{Proof}}
Let $\pi \in \mathcal{A}_{n,3}^\circ(\delta_3;1432)$.
By Lemma \ref{lem2.2}, we have $C(\pi)=(1,3,c_3,\ldots,n,2,\ldots,c_n)$ or $(1,3,c_3,\ldots,2,\ldots,n)$.
However, the one-line notation of the latter contains a $321$ pattern $\pi_12\pi_n=321$.

Now consider the one-line notation of $\pi$ when $C(\pi)=(1,3,c_3,\ldots,n,2,\ldots,c_n)$.
Since $\pi_1=3$ and $\pi_n=2$, all entries greater than $3$ must appear in increasing order.
If $c_n$ is odd, then the cycle decomposition of $\pi$ contains the cycle $(1,3,5,\ldots,c_n-2,c_n)$ which contradicts the fact that $\pi$ is cyclic.
Thus $c_n$ is even and $C(\pi)=(1,3,4,5,6,\ldots,n,2)$ or $(1,3,5,\ldots,n,2,4), (1,3,5,7,\ldots,n,2\allowbreak,4,6\allowbreak),\cdots$.
So
\[
a_{n,3}^\circ(\delta_3;1432)=\lfloor \frac{n-1}{2} \rfloor.
\]
$\hfill\square$

\begin{lem}\label{lem2.7}
\[
\sum_{j=4}^n a_{n,j}^\circ(\delta_3;1432)=\lfloor \frac{n-3}{2} \rfloor, \quad  \ n\geq 5.
\]
\end{lem}
\noindent \textbf{\textit{Proof}}
Let $\pi \in \mathcal{A}_{n,j}^\circ(\delta_3;1432)$ with $4\le j \le n$, then by Lemma \ref{lem2.2},
\[
C(\pi)=(1,j,c_3,\ldots,2,\ldots,3,\ldots,j-2,\ldots,j-1,\ldots,c_n).
\]
Moreover, the letters between $j$ and $j-1$ that are greater than $j$ must appear in increasing order; and the same holds for the letters after $2$ that are greater than $j$.

We show that the only possible case is
\[
C(\pi)=(1,j,2,j+1,3,j+2,\ldots,j-2,2j-3,2j-2,\ldots,n,j-1).
\]

First, we claim that for any $2 \le i \le j-2$, $i$ and $i+1$ are not adjacent in $C(\pi)$.
Suppose not.
Then there exists $i_0 \in [2,j-2]$ such that $\pi_{i_0}=i_0+1$.
Note that $c_n\geq j-1>i_0$, so we have $\pi_1\pi_{i_0}\pi_{c_n}=j(i_0+1)1$ as a $\delta_3$ pattern in the one-line notation of $\pi$.
A contradiction.

Subsequently, if $n$ occurs after $j-1$ in $C(\pi)$, then $c_n=n$ which leads to a $\delta_3$ pattern $\pi_12\pi_n=j21$ in one-line notation.
So $n$ occurs before $j-1$ and thus $C(\pi)=(1,j,c_3,\ldots,2,\allowbreak\ldots,j-2,\ldots,n,j-1)$.

If $c_3 \neq 2$, then $c_3=j+1$ since all the letters greater than $j$ appear in increasing order.
In this case, $\pi_{j-2}>j+1$ induces a $\delta_3$ pattern $\pi_{j-2}\pi_{j}\pi_{n}=\pi_{j-2}(j+1)(j-1)$ in one-line notation.
Thus $c_3=2$.

Finally, suppose that there are more than one letter between $i$ and $i+1$ in $C(\pi)$ for some $i \in [2,j-3]$.
Let $C(\pi)=(1,j,2,c_4,\ldots,i,j+r,j+r+1,\ldots,i+1,\ldots,j-2,\ldots,n,j-1)$ with $r \ge 1$, then $\pi_{j-2}\pi_{j+r}\pi_n=\pi_{j-2}(j+r+1)(j-1)$ forms a $\delta_3$ pattern in one-line notation.
Therefore, $C(\pi)=(1,j,2,j+1,3,j+2,\ldots,j-2,2j-3,2j-2,\ldots,n,j-1)$ which implies that $2j-3 \leq n$.

It is straightforward to check that chains
\[\{1<_{\pi}2<_{\pi}3<_{\pi}\cdots<_{\pi}j-2<_{\pi}2j-3<_{\pi}2j-2<_{\pi}\cdots<_{\pi}n-1\}\]
and
\[\{j-1<_{\pi}j<_{\pi}\cdots<_{\pi}2j-4<_{\pi}n\}\]
cover the poset $\mathcal{S}_\pi$.
So $\pi$ avoids $\delta_3$ in one-line notation by Lemma \ref{lem2.4}.

Therefore,
\[
\sum_{j=4}^n a_{n,j}^\circ(\delta_3;1432)=\sum_{j=4}^{\lfloor \frac{n+3}{2} \rfloor} 1=\lfloor \frac{n-3}{2} \rfloor.
\]
$\hfill\square$

\textbf{\textit{Proof of Theorem \ref{th1.1}}}
Combining equation \ref{eq1.1}, Proposition \ref{prop2.5}, Lemma \ref{lem2.6} and Lemma \ref{lem2.7}, we obtain that for $n\geq 5$
\[
a_n^\circ(\delta_3;1432)=a_{n-1}^\circ(\delta_3;1432)+\lfloor \frac{n-1}{2} \rfloor+\lfloor \frac{n-3}{2} \rfloor=\lfloor \frac{n^2+7}{2} \rfloor-2n.
\]
$\hfill\square$

\section{Proof of Theorem \ref{th1.2}}
In this section, we consider cyclic permutations that avoid $\delta_4=4321$ in one-line notation and avoid both $321$ and $2143$ in standard cycle form.
Similar to the proof of Theorem \ref{th1.1}, we enumerate permutations $\pi \in \mathcal{A}_n^\circ(\delta_4;1432)$ with $C(\pi)=(1,j,c_3,\ldots,c_n)$ according to the value of $j=c_2$.
\begin{lem}\label{lem3.1}
Let $n\ge6$ and $3\le j\le n-2$.
Then there is a bijection between
\[
\{\pi \in \mathcal{A}_n^\circ(\delta_4;1432) \mid C(\pi)=(1,j,\ldots,2,\ldots,3,\ldots,j-1,\ldots,n-1,n)\}
\]
and
\[
\{\pi' \in \mathcal{A}_{n-1}^\circ(\delta_4;1432) \mid C(\pi')=(1,j,\ldots,2,\ldots,3,\ldots,j-1,\ldots,n-1)\}
\]
\end{lem}
\noindent \textbf{\textit{Proof}}
Let $\pi\in \mathcal{A}_n^\circ(\delta_4;1432)$ with $C(\pi)=(1,j,\ldots,2,\ldots,3,\ldots,j-1,\ldots,n-1,n)$.
Define $\pi'$ by deleting the letter $n$ from $C(\pi)$, thus
$C(\pi')=(1,j,\ldots,2,\ldots,3,\ldots,j-1,\ldots,n-1)$.

Since $n$ is the last letter of $C(\pi)$, $n$ cannot participate in an occurrence of $321$ or $2143$ pattern.
Hence $C(\pi')$ avoids $321$ and $2143$.
Moreover, in one-line notation $n$ occurs at position $n-1$ in $\pi$, so $n$ cannot be part of a decreasing
subsequence of length $4$; therefore $\pi'$ avoids $\delta_4$.
Thus $\pi'\in \mathcal{A}_{n-1}^\circ(\delta_4;1432)$ and $C(\pi')$ has the required form.

Conversely, let $\pi'\in \mathcal{A}_{n-1}^\circ(\delta_4;1432)$ with $C(\pi')=(1,j,\ldots,2,\ldots,3,\ldots,j-1,\ldots,n-1)$.
Construct $\pi$ by inserting the letter $n$ at the end of $C(\pi')$, i.e., $C(\pi)=(1,j,\ldots,2,\ldots,3,\ldots,j-1,\ldots,n-1,n)$.
By the same reasoning as above, inserting $n$ at the end cannot create an occurrence of $321$ or $2143$ in the standard cycle form, and inserting
$n$ at position $n-1$ in one-line notation cannot create an occurrence of $\delta_4$.
Hence $\pi\in \mathcal{A}_n^\circ(\delta_4;1432)$.

The two constructions are inverse to each other, so they define the desired bijection.
\hfill$\square$

\begin{prop}\label{prop3.2}
For $n\geq 6$, $a_{n,2}^\circ(\delta_4;1432)= a_{n-1}^\circ(\delta_4;1432)$.
\end{prop}
\noindent \textbf{\textit{Proof}}
This follows immediately from Lemma \ref{lem2.3}.
$\hfill\square$

\begin{lem}\label{lem3.3}
For $n\geq 6$, $a_{n,3}^\circ(\delta_4;1432)=2^{n-4}+\lfloor \frac{n-2}{2} \rfloor$.
\end{lem}
\noindent \textbf{\textit{Proof}}
Let $\pi \in \mathcal{A}_{n,3}^\circ(\delta_4;1432)$ with $C(\pi)=(1,3,\ldots,2,\ldots)$.
By Lemma \ref{lem2.2}, the letters between $3$ and $2$ appear in increasing order, and the same holds for the letters after $2$.
Thus $C(\pi)$ falls into one of the following cases.

\noindent(i)   $C(\pi)=(1,3,\ldots,n-1,2,\ldots,n)$

Consider the one-line notation.
Since $\pi_{n-1}=2$, $\pi_{n}=1$, $\pi$ has to be $345\cdots n21$.
If $n$ is odd, then the cycle form of $\pi$ contains $(1,3,5,\ldots,n)$, contradicting that $\pi$ is cyclic.
If $n$ is even, then there is exactly one $\pi$ with $C(\pi)=(1,3,5,\ldots,n-3,n-1,2,4,\ldots,n-2,n)$.
Thus, the number of such permutations is $\frac{1+(-1)^n}{2}$.

\noindent(ii)  $C(\pi)=(1,3,\ldots,2,\ldots,n-1,n)$

Define $b_n=|\{\pi\in \mathcal{A}_{n,3}^\circ(\delta_4;1432) \mid C(\pi)=(1,3,\ldots,2,\ldots,n-1,n)\}|$.
By Lemma \ref{lem3.1}, we have the recurrence
\[
b_n=b_{n-1}+\frac{1+(-1)^{n-1}}{2}.
\]
Since $b_5=1$, we get $b_n=\lfloor \frac{n-3}{2} \rfloor$.

\noindent(iii) $C(\pi)=(1,3,c_3,\ldots,n,2,\ldots,c_n)$

It is straightforward to check that chains
\[\{1<_{\pi}3<_{\pi}c_3<_{\pi}\cdots<_{\pi}\pi^{-1}(n)\},\]

\[\{2<_{\pi}\pi_2<_{\pi}\cdots<_{\pi}c_{n-2}<_{\pi}c_{n-1}\}\]
and
\[\{c_n<_{\pi} n\}\]
cover the poset $\mathcal{S}_\pi$.
Therefore, $\pi$ avoids $\delta_4$ in one-line notation by Lemma \ref{lem2.4}.
Each of the remaining $n-4$ letters can be placed in one of two ways: (i) inserted (in increasing order) between $3$ and $n$, or (ii) placed (in increasing order) after $2$.
Thus there are $2^{n-4}$ such cyclic permutations.

So we conclude that $a_{n,3}^\circ(\delta_4;1432)=2^{n-4}+\frac{1+(-1)^n}{2}+\lfloor \frac{n-3}{2} \rfloor=2^{n-4}+\lfloor \frac{n-2}{2} \rfloor$.
$\hfill\square$

\begin{lem}\label{lem3.4}
For $n\geq6$, $\sum_{j=4}^n a_{n,j}^\circ(\delta_4;1432)=3\times2^{n-4}-n+\lfloor \frac{n}{2} \rfloor$.
\end{lem}
\noindent \textbf{\textit{Proof}}
Let $\pi \in \mathcal{A}_{n,j}^\circ(\delta_4;1432)$ with $4\le j \le n$, then by Lemma \ref{lem2.2},
\[
C(\pi)=(1,j,c_3,\ldots,2,\ldots,3,\ldots,j-2,\ldots,j-1,\ldots,c_n).
\]
Moreover, the letters between $j$ and $j-1$ that are greater than $j$ must appear in increasing order; the same holds for the letters after $2$ that are greater than $j$.

Therefore, for $4 \le j \le n-2$, $C(\pi)$ is one of the following four cases.

\noindent (i)   $C(\pi)=(1,j,\ldots,2,\ldots,i,\ldots,n-1,i+1,i+2,\ldots,j-1,n), \ 1\leq i\leq j-2$

We show that the only possibility is $C(\pi)=(1,j,2,j+1,3,j+2,\ldots,j-2,n-1,j-1,n)$.

Note that for any $m \in [2,j-2]$, $m$ and $m+1$ cannot be adjacent in $C(\pi)$.
Suppose not.
Then there exists $m_0 \in [2,j-2]$ such that $\pi_{m_0}=m_0+1$.
In this case, $\pi_1\pi_{m_0}2\pi_n=j(m_0+1)21$ would be a $\delta_4$ pattern in one-line notation (note that the preimage of $2$ is not less than $j$).
Thus we have $i+1=j-1$ which implies that $C(\pi)=(1,j,\ldots,2,\ldots,j-2,\ldots,n-1,j-1,n)$.

Consider the one-line notation of $\pi$, i.e. $\pi=j\cdots n\cdots(j-1)1$.
It is clear that all the entries that lie in $[j+1,n-1]$ occur before $n$ and appear in increasing order.
Likewise, all the entries lying in $[2,j-2]$ occur after $n$ and are in increasing order in the one-line notation, since $m$ and $m+1$ are not adjacent when $m \in [2,j-2]$ in $C(\pi)$.
Therefore, $\pi=j(j+1)\cdots n23\cdots(j-1)1$ with $C(\pi)=(1,j,2,j+1,3,j+2,\ldots,j-2,n-1,j-1,n)$ which implies $j= \frac{n+2}{2}$.
So the number of such cyclic permutations is $\frac{1+(-1)^n}{2}$ as $j$ ranges from $4$ to $n-2$.

\noindent (ii)  $C(\pi)=(1,j,\ldots,2,\ldots,j-1,\ldots,n-1,n)$

Define $d_n=\sum\limits_{j=4}^n|\{\pi\in \mathcal{A}_{n,j}^\circ(\delta_4;1432) \mid C(\pi)=(1,j,\ldots,2,\ldots,j-1,\ldots,n-1,n)\}|$.
By Lemma \ref{lem3.1}, we have the recurrence
\[
d_n=d_{n-1}+\frac{1+(-1)^n}{2}.
\]
Since $d_6=0$, we get $d_n=\lfloor \frac{n-5}{2} \rfloor$.

\noindent (iii) $C(\pi)=(1,j,c_3,\ldots,n,2,3,\ldots,j-1,\ldots,c_n)$

In this case, the letters between $j$ and $n$ appear in increasing order; the same holds for the letters after $j-1$.
Since
\[\{1<_{\pi} j<_{\pi}c_3<_{\pi}c_4<_{\pi}\cdots<_{\pi}\pi^{-1}(n)\},\]

\[\{2<_{\pi}3<_{\pi}\cdots<_{\pi}j-1<_{\pi}\pi_{j-1}<_{\pi}\cdots<_{\pi}c_{n-1}\}\]
and
\[\{c_n<_{\pi}n\}\]
cover the poset $\mathcal{S}_\pi$, $\pi$ avoids $\delta_4$ in one-line notation by Lemma \ref{lem2.4}.
There are $2^{n-j-1}$ such cyclic permutations.

\noindent (iv)  $C(\pi)=(1,j,c_3,\ldots,2,\ldots,i,\ldots,n,i+1,i+2,\ldots,j-1)$ for some $i \in [2,j-2]$

In this case, all the letters lying in $[j+1,n-1]$ occur between $j$ and $n$ and are in increasing order.

Let $r_1<r_2<\cdots<r_p$ be the elements of $[2,i]$ whose images under $\pi$ are less than $i+1$, and let $s_1<\cdots<s_{i-1-p}$ be the remaining elements of $[2,i]$.
Likewise, let $u_1<u_2<\cdots<u_q$ be the elements of $[j,n-1]$ whose images under $\pi$ are less than $i+1$, and let $v_1<\cdots<v_{n-j-q}$ be the remaining elements of $[j,n-1]$.

Then chains
\[\{1<_{\pi}s_1<_{\pi}s_2<_{\pi}\cdots<_{\pi}s_{i-1-p}\},\]

\[\{r_1<_{\pi}r_2<_{\pi}\cdots<_{\pi}r_p<_{\pi}i+1<_{\pi}\cdots<_{\pi}j-2<_{\pi}v_1<_{\pi}v_2<_{\pi}\cdots<_{\pi}v_{n-j-q}\}\]
and
\[\{j-1<_{\pi}u_1<_{\pi}\cdots<_{\pi}u_q<_{\pi}n\}\]
cover $\mathcal{S}_{\pi}$ and thus $\pi$ avoids $\delta_4$ in one-line notation by Lemma \ref{lem2.4}.
It is not hard to see that the number of such $C(\pi)$ equals the number of non-negative solutions to $x_1+x_2+\cdots+x_{j-2}=n-j$ minus $1$.
Namely $\binom{n-3}{j-3}-1$ (the solution $x_1=n-j$, $x_2=x_3=\cdots=x_{j-2}=0$ corresponds to the case $C(\pi)=(1,j,j+1,\ldots,n,2,3,\ldots,j-1)$, which is classified as case (iii)).

Note that the proofs of cases (iii) and (iv) also apply when $j=n-1$.
Consequently, among the resulting $n-2$ cyclic permutations, all avoid $\delta_4$ in one-line notation except for $(1,n-1,2,3,\ldots,n-2,n)$.

For $j=n$, there is only one permutation in $\mathcal{A}_{n,n}^\circ(\delta_4;1432)$, which is $(1,n,2,3,\ldots,n-1)$.

Therefore, we conclude that
\begin{align*}
\sum_{j=4}^n a_{n,j}^\circ(\delta_4;1432)
&=\frac{1+(-1)^n}{2}+\lfloor \frac{n-5}{2} \rfloor+\sum_{j=4}^{n-2} [2^{n-j-1}+\binom{n-3}{j-3}-1]+n-3+1\\
&= 3\times2^{n-4}-n+\lfloor \frac{n}{2} \rfloor.
\end{align*}
$\hfill\square$

\textbf{\textit{Proof of Theorem \ref{th1.2}}}
Combining equation \ref{eq1.1}, Proposition \ref{prop3.2}, Lemma \ref{lem3.3} and Lemma \ref{lem3.4}, we obtain
\begin{align*}
a_n^\circ(\delta_4;1432)&= a_{n-1}^\circ(\delta_4;1432)+2^{n-4}+\lfloor \frac{n-2}{2} \rfloor+3\times2^{n-4}-n+\lfloor \frac{n}{2} \rfloor\\
&= 2^{n-1}-\lfloor \frac{3n-5}{2} \rfloor, \quad n\geq6.
\end{align*}
For $n=5$, it is straightforward to verify that $a_5^\circ(\delta_4;1432)=11$.
$\hfill\square$

\section{Proof of Theorem \ref{th1.3}}

\begin{lem}\label{lem4.1}
Let $n \geq 5$ and $\pi\in\mathfrak{S}_n$ be cyclic.
If the standard cycle form $C(\pi)$ avoids both $321$ and $2143$, then the one-line notation of $\pi$ avoids $\delta_5$.
Consequently, $\pi$ avoids $\delta_k$ for all $k\ge5$ in one-line notation.
\end{lem}

\noindent \textbf{\textit{Proof}}
It suffices to show that if there exist $1 \le x_1 < x_2 < x_3 < x_4 < x_5 \le n$ such that $y_5 > y_4 > y_3 > y_2 > y_1$, then $C(\pi)$ contains a $321$ or $2143$ pattern, where $y_{6-i}=\pi_{x_i}$.

Assume $y_1>1$. Then there exists a permutation $\tau\in\mathfrak{S}_5$ such that
\[
C(\pi)=(1,\ldots,x_{\tau_1}, y_{6-\tau_1},\ldots,x_{\tau_2}, y_{6-\tau_2},\ldots,x_{\tau_3}, y_{6-\tau_3},\ldots,x_{\tau_4}, y_{6-\tau_4},\ldots,x_{\tau_5}, y_{6-\tau_5},\ldots)
\]
with possibly $1=x_{\tau_1}$.

Since $C(\pi)$ avoids $321$, the sequence $x_{\tau_1}x_{\tau_2}x_{\tau_3}x_{\tau_4}x_{\tau_5}$ has to avoid $321$.
Moreover, $y_{6-\tau_1}\allowbreak y_{6-\tau_2}y_{6-\tau_3}y_{6-\tau_4}\allowbreak y_{6-\tau_5}$ has to also avoid $321$, which implies that $x_{\tau_1}x_{\tau_2}x_{\tau_3}x_{\tau_4}x_{\tau_5}$ avoids $123$.
Therefore, $\tau \in \mathfrak{S}_5$ avoids both $321$ and $123$.
However, this is impossible by the Erd\H{o}s--Szekeres theorem \cite{Erdos} or direct verification.

Now, assume $y_1=1$ so that
\[
C(\pi)=(1,\ldots,x_{\sigma_1}, y_{6-\sigma_1},\ldots,x_{\sigma_2}, y_{6-\sigma_2},\ldots,x_{\sigma_3}, y_{6-\sigma_3},\ldots,x_{\sigma_4}, y_{6-\sigma_4},\ldots,x_5)
\]
for some $\sigma \in \mathfrak{S}_4$.

As noted above, $\sigma$ has to avoid both $321$ and $123$.
The only four such permutations in $\mathfrak{S}_4$ are $2143,2413,3142$ and $3412$.
We show that for any $\sigma \in \{2143,2413,3142,\allowbreak3412\}$, $C(\pi)$ contains either a $321$ or a $2143$ pattern.

\noindent (i)   $\sigma=2143$

It is evident that $x_2x_1x_4x_3$ forms a $2143$ pattern.

\noindent (ii)  $\sigma=2413$

In this case, $C(\pi)=(1,\ldots,x_2,y_4,\ldots,x_4,y_2,\ldots,x_1,y_5,\ldots,x_3,y_3,\ldots,x_5)$.
If $y_5>x_4$, then $x_2x_1y_5x_3$ forms a $2143$ pattern; otherwise $x_4y_5y_3$ forms a $321$ pattern.

\noindent (iii) $\sigma=3142$

In this case, $C(\pi)=(1,\ldots,x_3,y_3,\ldots,x_1,y_5,\ldots,x_4,y_2,\ldots,x_2,y_4,\ldots,x_5)$.
If $y_2>x_1$, then $y_3x_1y_5y_4$ forms a $2143$ pattern; otherwise $x_3x_1y_2$ forms a $321$ pattern.

\noindent (iv) $\sigma=3412$

It is evident that $y_3y_2y_5y_4$ forms a $2143$ pattern.
$\hfill\square$

\noindent \textbf{\textit{Proof of Theorem \ref{th1.3}}}
In \cite{Callan,Vella}, Callan and Vella found that the total number of cyclic permutations in $\mathfrak{S}_n$ all of whose standard cycle forms avoid $321$ and $2143$ is equal to $2^n+1-2n-\binom{n}{3}$.
Then the result follows immediately from Lemma \ref{lem4.1}.
$\hfill\square$

In conclusion, we derived the explicit formula for $a_n^\circ(\delta_k;1432)$ for all $n\geq5$ and $k\geq3$.

\section*{Acknowledgments}
The authors thank Qiongqiong Pan for proposing the problem and Feng Zhao for helpful discussions.

%\section{Reference}


\begin{thebibliography}{99}
\bibitem{Archer1}
K. Archer, E. Borsh, J. Bridges, C. Graves, M. Jeske,
Cyclic permutations avoiding patterns in both one-line and cycle forms,
Australas. J. Combin. 94 (1) (2026) 50--92.
\bibitem{Archer2}
K. Archer, E. Borsh, J. Bridges, C. Graves, M. Jeske,
Pattern-restricted one-cycle permutations with a pattern-restricted cycle form,
Enumer. Comb. Appl. 5 (1) (2025) S2R3.
\url{https://doi.org/10.54550/ECA2025V5S1R3}.
\bibitem{Callan}
D. Callan,
Pattern avoidance in circular permutations,
arXiv:math/0210014, 2002, preprint.
\url{https://doi.org/10.48550/arXiv.math/0210014}.
\bibitem{Dilworth}
R.P. Dilworth,
A decomposition theorem for partially ordered sets,
Ann. Math. (2) 51 (1) (1950) 161--166.
\url{https://doi.org/10.2307/1969503}.
\bibitem{Erdos}
P. Erd\H{o}s, G. Szekeres,
A combinatorial problem in geometry,
Compos. Math. 2 (1935) 463--470.
\bibitem{Knuth}
D.E. Knuth,
The Art of Computer Programming, Vol.~1: Fundamental Algorithms,
Addison-Wesley, Reading, MA, 1968.
\bibitem{PanJ}
J. Pan,
On a conjecture about pattern avoidance of cyclic permutations,
Bull. Malays. Math. Sci. Soc. 48 (2025) 77.
\url{https://doi.org/10.1007/s40840-025-01859-9}.
\bibitem{PanQ}
Q. Pan,
Pattern avoidance of cyclic permutations,
Bull. Malays. Math. Sci. Soc. 48 (2025) 213.
\url{https://doi.org/10.1007/s40840-025-01998-z}.
\bibitem{Simion}
R. Simion, F.W. Schmidt,
Restricted permutations,
Eur. J. Comb. 6 (4) (1985) 383--406.
\url{https://doi.org/10.1016/S0195-6698(85)80052-4}.
\bibitem{Vella}
A. Vella,
Pattern avoidance in permutations: linear and cyclic orders,
Electron. J. Comb. 9 (2) (2003) R18.
\url{https://doi.org/10.37236/1690}.
\end{thebibliography}
\end{document}